\documentclass[12pt]{amsart}
\usepackage{latexsym}
\usepackage{amssymb}
\usepackage{amscd}
\usepackage{mathrsfs}
\usepackage{bbm}
\usepackage{CJK}
\usepackage[all]{xy}
\usepackage{tikz-cd}

\renewcommand\a{\alpha}
\renewcommand\b{\beta}
\newcommand\g{\gamma}
\renewcommand\d{\delta}

\newcommand\io{\iota}

\newcommand\s{\sigma}

\newcommand\vf{\varphi}

\newcommand\vD{\varDelta}

\newcommand\ve{\varepsilon}

\newcommand{\QQ}{\mathbb Q}
\newcommand{\FF}{\mathbb F}

\newcommand{\ZZ}{\mathbb Z}
\newcommand{\NN}{\mathbb N}

\newcommand\BA{\mathbf A}

\newcommand\BC{\mathbf C}

\newcommand\BZ{\mathbf Z}

\newcommand\BJ{\mathbf J}
\newcommand\bB{\mathbf B}

\newcommand\BU{\mathbf U}
\newcommand\BV{\mathbf V}

\newcommand\Bh{\mathbf h}

\newcommand\Bn{\mathbf n}

\newcommand\Bc{\mathbf c}
\newcommand\Bd{\mathbf d}

\newcommand\Bf{\mathbf f}

\newcommand\Bx{\mathbf x}

\newcommand\CB{\mathcal{B}}

\newcommand\SL{\mathscr{L}}

\newcommand\SX{\mathscr{X}}

\newcommand\Fg{\mathfrak g}

\newcommand\Fs{\mathfrak s}
\newcommand\Fl{\mathfrak l}

\newcommand\iv{^{-1}}

\newcommand\wt{\widetilde}

\newcommand\ol{\overline}
\newcommand\ul{\underline}

\newcommand\lra{\leftrightarrow}

\newcommand\id{\operatorname{id}}

\newcommand\weit{\operatorname{wt}}

\newcommand{\isom}{\,\raise2pt\hbox{$\underrightarrow{\sim}$}\,}
\numberwithin{equation}{section}

\newtheorem{thm}{Theorem}[section]
\newtheorem{lem}[thm]{Lemma}

\newtheorem{prop}[thm]{Proposition}

\def \para#1{\par\medskip\textbf{#1}
              \addtocounter{thm}{1}}

\def \remark#1{\par\medskip\noindent
                \textbf{Remark #1}
                \addtocounter{thm}{1}}

\def \remarks#1{\par\medskip\noindent
                \textbf{Remarks #1}
                \addtocounter{thm}{1}}

\begin{document}
\setlength{\baselineskip}{4.9mm}
\setlength{\abovedisplayskip}{4.5mm}
\setlength{\belowdisplayskip}{4.5mm}
\renewcommand{\theenumi}{\roman{enumi}}
\renewcommand{\labelenumi}{(\theenumi)}
\renewcommand{\thefootnote}{\fnsymbol{footnote}}
\renewcommand{\thefootnote}{\fnsymbol{footnote}}
\parindent=20pt
\medskip
\begin{center}
  {\bf Elementary construction of canonical bases, \\
  foldings, and piecewise linear bijections} 
\par
\vspace{1cm}
Toshiaki Shoji and Zhiping Zhou
\\
\vspace{0.7cm}
\title{}
\end{center}

\begin{abstract} 
Let $\BU_q^-$ be the negative half of the quantum group of finite type.
We construct the canonical basis of $\BU_q^-$ by applying the folding theory
of quantum groups and piecewise linear parametrization of
canonical basis. Our construction is elementary,
in the sense
that we don't appeal to Lusztig's geometric theory of canonical bases,
nor to Kashiwara's theory of crystal bases. 
\end{abstract}

\maketitle
\pagestyle{myheadings}
\markboth{SHOJI AND ZHOU}{CANONICAL BASES}

\begin{center}
{\sc Introduction}
\end{center}

\par\medskip
Let $\BU_q^-$ be the negative half of the quantum group $\BU_q$, associated to
the Cartan datum $X$. 
In the case where $\BU_q^-$ is symmetric of finite type, Lusztig
constructed in \cite{L-can} the canonical basis of $\BU_q^-$, by making use of the
representation theory of quivers (but not using the geometric theory of quivers mentioned below).
After that in \cite{L-per}, he constructed the canonical basis for
$\BU_q^-$ of symmetric, Kac-Moody type, by making use
of the geometric theory of quivers, namely, by using
the categorification in terms of a certain category of perverse sheaves
on the representation space of the quiver associated to $X$. 
\par
This result was generalized in his book \cite{L-book},
to the non-symmetric case as follows.
Assume that $\BU_q^-$ is symmetric, and $\s : \BU_q^- \to \BU_q^-$ is the 
automorphism of quantum groups
induced from the diagram automorphism $\s$ on the Cartan datum $X$.
Then $\s$ induces a Cartan datum $\ul X$, and one can define the quantum group
$\ul\BU_q^-$ associated to $\ul X$. Any quantum group of Kac-Moody type is obtained
as $\ul\BU_q^-$ from a certain quantum group $\BU_q^-$ of symmetric type with automorphism $\s$. 
He constructed the canonical basis of $\ul\BU_q^-$, by making use of the category
of perverse sheaves associated to $\BU_q^-$ as above, but with the action of $\s$ on
that category.
In general, a folding theory of quantum groups discusses the relationship between
$\BU_q^-$ with automorphism $\s$, and $\ul\BU_q^-$. 
Hence the above result is regarded as a folding theory for the category with
an automorphism.
\par
On the other hand, Kashiwara established in \cite{K} the theory of crystal bases,
and constructed the global crystal basis for any $\BU_q^-$ of Kac-Moody type.
It is known that Lusztig's canonical basis
coincides with Kashiwara's global crystal basis (\cite{GL} for the symmetric case,
\cite{MSZ2} for the general case).
\par
For $\BU_q^-$ of finite type, there is an attempt to
construct the canonical basis in a more elementary way, by making use of the
PBW bases of $\BU_q^-$.  
This problem is reduced to the case where $\BU_q^-$ has rank 2, and it is easily
solved if $\BU_q^-$ is symmetric, namely, $\BU_q^-$ is of type $A_1 \times A_1$
or $A_2$.
In the case where $\BU_q^-$ is not symmetric,
namely $\BU_q^-$ is of type $B_2$ or $G_2$,  the canonical bases were constructed by
Lusztig \cite{L-root}, Xi \cite{X1}, \cite{X2}.  But the computations are rather
hard, and in particular, it is not yet complete in the case of type $G_2$.
In that case, in part,
one has to appeal to the general theory of canonical basis or crystal basis. 
\par
In this paper, we assume that $\BU_q^-$ is of finite type.
We construct the canonical basis of $\BU_q^-$ in an elementary way, in the sense
that we don't appeal to Lusztig's theory of canonical basis nor to Kashiwara's theory
of crystal basis. Also our approach is different from that of Lusztig and Xi.
Our main ingredient is the folding theory of quantum groups developed in
\cite{SZ1},\cite{MSZ1}. This theory is an analogue of Lusztig's theory of
the category with automorphism,
but we compare $\BU_q^-$ and $\ul\BU_q^-$ directly (not on the category),
and the discussion is more elementary.
\par
The construction of the canonical basis is finally reduced to the case of type $B_2$. 
The canonical basis is parametrized by the PBW basis once it is fixed. Since the PBW basis
is not unique, one needs to compare the labeling of the canonical basis for
a different choice of PBW bases.  This difference of the labeling is
described by a certain map called the piecewise linear bijection.
In the case of type $B_2$, this map was explicitly computed by Lusztig \cite{L-piece}.
By making use of this map, we can determine
the canonical basis for $\BU_q^-$ of type $B_2$.

\bigskip
\medskip

\section{Review on quantum groups}

\para{1.1.}
Let $X = (I, (\ ,\ ))$ be a Cartan datum of finite type, where
$I$ is a finite set, and $(\ ,\ )$ is a symmetric bilinear form on the
$\QQ$-vector space $E = \bigoplus_{i \in I}\QQ \a_i$ with basis $\a_i$, satisfying
the property
\par\medskip
\par \ $(\a_i,\a_i) \in 2\ZZ_{> 0}$ for any $i \in I$,  \\
\par \ $\frac{2(\a_i,\a_j)}{(\a_i,\a_i)} \in \ZZ_{\le 0}$  for any $i \ne j \in I$.
\par\medskip
The Cartan datum $X$ is called symmetric if $(\a_i,\a_i) = 2$ for any $i \in I$,
and $(\a_i,\a_j) \in \{ 0, -1\}$ for any $i \ne j \in I$.  
Let $Q = \bigoplus_{i \in I}\ZZ \a_i$ be the root lattice of $X$.  We set
$Q_+ = \sum_{i \in I}\NN \a_i$, and $Q_- = -Q_+$.  

\para{1.2.}
Let $q$ be an indeterminate, and for an integer $n$, a positive integer $m$, set
\begin{equation*}
[n] = \frac{q^n- q^{-n}}{q - q\iv}, \qquad [m]^! = [1][2]\cdots [m], \quad [0]^! = 1.
\end{equation*}
For each $i \in I$, set $q_i = q^{(\a_i,\a_i)/2}$.  We denote by $[n]_i$ 
the element obtained from $[n]$ by replacing $q$ by $q_i$.  $[m]^!_i$ is defined similarly.
\par
For the Cartan datum $X$, let $\BU_q^- = \BU_q^-(X)$ be the negative half of
the quantum group $\BU_q = \BU_q(\Fg)$, where $\Fg$ is a semisimple
Lie algebra over $\BC$ associated to $X$.  $\BU_q^-$ is an associative algebra over
$\QQ(q)$ generated by $f_i (i \in I)$ satisfying the $q$-Serre relations
\begin{equation*}
\tag{1.2.1}  
\sum_{k + k' = 1-a_{ij}}(-1)^kf_i^{(k)}f_jf_i^{(k')} = 0 \quad\text{ for } i \ne j \in I,
\end{equation*}
where we set $a_{ij} = 2(\a_i,\a_j)/(\a_i,\a_i)$ and
$f_i^{(n)} = f_i^n/[n]_i^!$ for any $n \in \NN$. 
\par
Set $\BA = \ZZ[q,q\iv]$,and let ${}_{\BA}\BU_q^-$ be Lusztig's integral form on $\BU_q^-$,
namely, the $\BA$-subalgebra of $\BU_q^-$ generated by $f_i^{(n)}$ for $i \in I, n \in \NN$.
\par
$\BU_q^-$ has a weight space decomposition $\bigoplus_{\g \in Q_-}(\BU_q^-)_{\g}$,
where $(\BU_q^-)_{\g}$ is a subspace of $\BU_q^-$ spanned by elements $f_{i_1}\cdots f_{i_k}$
such that $\a_{i_1} + \cdots + \a_{i_k} = -\g$. 
Hence $\dim (\BU_q^-)_{\g} < \infty$ for any $\g$. 
If $x \in (\BU_q^-)_{\g}$, then $x$ is said to be homogeneous with weight $\g$,
and is written as $\weit(x) = \g$.

\para{1.3.}
Let $\vD$ be the root system associated to $X$, and $\vD^+$ the set of
positive roots with the set of simple roots $\{ \a_i \mid i \in I\}$. 
Let $W$ be the Weyl group of $X$ which is generated by
simple reflections $s_i$ corresponding to $\a_i$.
Let $w_0$ be the longest element in $W$, and
let $w_0 = s_{i_1}s_{i_2}\cdots s_{i_N}$ be a reduced expression of $w_0$,
where $N = |\vD^+|$.  We set $\Bh = (i_1, \dots, i_N)$, and denote it
a reduced sequence of $w_0$. We fix a reduced sequence $\Bh$.  For each
$1 \le k \le N$, set $\b_k = s_{i_1}s_{i_2}\cdots s_{i_{k-1}}(\a_{i_k})$.
It is known that $\b_k \in \vD^+$, and $\b_1, \dots, \b_N$ are all distinct.
Thus $\vD^+ = \{ \b_1, \dots, \b_N\}$, and $\Bh$ determines a total order on
the set $\vD^+$.
\par
Let $T_i : \BU_q \to \BU_q$ be the braid group action, and set
\begin{equation*}
f_{\b_k}^{(n)} = T_{i_1}T_{i_2}\cdots T_{i_{k-1}}(f_{i_k}^{(n)}).
\end{equation*}
Then $f_{\b_k}^{(n)} \in \BU_q^-$ is homogeneous with $\weit(f_{\b_k}^{(n)}) = n\b_k$.
$f_{\b_k}^{(n)}$ are called root vectors in $\BU_q^-$ corresponding to
the root $\b_k$.  For $\Bc = (c_1, \dots, c_N) \in \NN^N$,
we define
\begin{equation*}
L(\Bc, \Bh) = f_{\b_1}^{(c_1)}f_{\b_2}^{(c_2)}\cdots f_{\b_N}^{(c_N)}.
\end{equation*}  
Then the set $\SX_{\Bh} = \{ L(\Bc, \Bh) \mid \Bc \in \NN^N\}$ gives a basis
of $\BU_q^-$, which is called the PBW basis of $\BU_q^-$ associated to $\Bh$.
The following result is known.
For the proof, see Remark 1.6 below. 

\begin{prop}  
Let $\Bh$ be a reduced sequence for $w_0$. 
\begin{enumerate}
\item \ The $\ZZ[q]$-submodule of $\BU_q^-$  spanned by $\SX_{\Bh}$
does not depend on the choice of $\Bh$, which we denote by
$\SL_{\ZZ}(\infty)$. Then $\SX_{\Bh}$ gives an $\ZZ[q]$-basis of $\SL_{\ZZ}(\infty)$. 

\item \ The $\ZZ$-basis of $\SL_{\ZZ}(\infty)/q\SL_{\ZZ}(\infty)$ obtained from $\SX_{\Bh}$
does not depend on the choice of $\Bh$.  

\item \ For any reduced sequence $\Bh$, $\SX_{\Bh}$ gives an $\BA$-basis of ${}_{\BA}\BU_q^-$. 
\end{enumerate}  
\end{prop}

\para{1.5.}
The bar involution ${}^- : \BU_q^- \to \BU_q^-$ is defined as an $\QQ$-algebra
automorphism given by $q \mapsto q\iv, f_i \mapsto f_i \ (i \in I)$.
We define a total order on $\NN^N$ by $\Bc < \Bc'$, for
$\Bc = (c_1, \dots, c_N), \Bc' = (c_1', \dots, c_N')$,  if 
there exists $k$ such that $c_1 = c_1', \dots, c_{k-1} = c_{k-1}'$ and $c_k < c_k'$. 
The following triangularity property can be proved by a standard argument
(assuming that $\SX_{\Bh}$ gives an $\BA$-basis of ${}_{\BA}\BU_q^-$). 

\begin{equation*}
\tag{1.5.1}  
  \ol{ L(\Bc, \Bh)} = L(\Bc,\Bh) + \sum_{\Bc' > \Bc}a_{\Bc', \Bc}L(\Bc',\Bh)
       \quad\text{ with } \quad a_{\Bc',\Bc} \in \BA. 
\end{equation*}  
By making use of (1.5.1), the following result is easily verified.
\par\medskip\noindent
(1.5.2) \ Assume that $\SX_{\Bh}$ gives an $\BA$-basis of ${}_{\BA}\BU_q^-$.
For a given $L(\Bc,\Bh) \in \SX_{\Bh}$, there exists a unique element
$b(\Bc, \Bh)$ satisfying the properties
\par\medskip
\begin{enumerate}
\item \
$\ol{b(\Bc,\Bh)} = b(\Bc,\Bh)$, 
\item \
 $b(\Bc, \Bh) = L(\Bc,\Bh) + \sum_{\Bc' > \Bc}p_{\Bc',\Bc}L(\Bc', \Bh),
\quad \text{ with }\quad  p_{\Bc,\Bc'} \in q\ZZ[q]$.
\end{enumerate}
\par
Then $\bB_{\Bh} = \{ b(\Bc,\Bh) \mid \Bc \in \NN^N\}$ gives a basis of $\BU_q^-$,
and an $\BA$-basis of ${}_{\BA}\BU_q^-$.
In particular, if Proposition 1.4 holds, then $\bB_{\Bh}$ exists, and does not depend on
the choice of $\Bh$ (by Proposition 1.4 (ii)),
which is denoted by $\bB$, and is called the canonical basis of
$\BU_q^-$. 

\remarks{1.6.}
(i) \ 
We give a historical remark on the proof of Proposition 1.4.
In the case where $\BU_q^-$ is symmetric, the properties (i) $\sim$ (iii) of
Proposition 1.4 
was proved by \cite[42]{L-book} by an elementary argument. The proof is also reduced to
the rank 2 case, namely, $X$ is of type $A_1 \times A_1$ or $A_2$, 
in which case the proof is easy.
\par
Now assume that $\BU_q^-$ is not symmetric.  
In this case the property (iii) was proved by \cite{L-root}, after reducing to the
rank 2 case, namely, $X$ is of type $B_2$ or $G_2$.
In the case of $B_2$, this was proved by computing the commutation relations for
root vectors.  The computation for the case $G_2$ is more complicated,
and it was shown by imitating the discussion of Kostant
concerning the $\BZ$-form of the universal enveloping algebra $\BU(\Fg)$.  Later,
Xi \cite{X1} proved it by a similar method as in the case of $B_2$.
\par
In \cite{X2}, Xi determined the canonical basis of $\BU_q^-$ in the case of $B_2$,
as an explicit linear combination of the PBW basis. Through this computation,
he verified the properties (i) $\sim$ (iii) in the proposition.
In the case of type $G_2$, constructing the canonical basis as a linear combination
of the PBW basis seems to be much more complicated, and as far as the authors know,
it is not yet done.
Thus properties (i) and (ii) are not yet verified from this approach. 
\par
In the general setup, the property (i) was proved by Lusztig \cite{L-book}
by applying the geometric theory of canonical bases.  The property (ii) was
proved by Saito \cite{S} by applying
the theory of crystal bases due to Kashiwara \cite{K}.
Note that in those discussions, the existence of canonical bases or crystal bases
was first proved, then Proposition 1.4 was proved by making use of canonical bases or
crystal bases.
\par
(ii) \ 
In tis paper, we are interested in constructing the canonical bases in an elementary
way, in the sense that we don't appeal to Lusztig's geometric theory of canonical bases
nor to Kashiwara's theory of crystal bases.

\para{1.7.}
There exists a symmetric bilinear form $(\ ,\ )$ on $\BU_q^-$, which is induced from
the symmetric bilinear form  on the free algebra $'\Bf$ defined in \cite[Prop. 1.2.3]{L-book}.
It is known that the PBW basis $\SX_{\Bh}$ gives an orthogonal basis of $\BU_q^-$
with respect to the bilinear form $(\ ,\ )$.
\par
The $\ZZ[q]$-submodule $\SL_{\ZZ}(\infty)$ of $\BU_q^-$ defined in Proposition 1.4
can be redefined by using the bilinear form $(\ ,\ )$ as follows.
Set $\BA_0 = \QQ[[q]]\cap \QQ(q)$. Then $\BA_0$ is a local ring with the maximal ideal
$q\BA_0$, and $\BA_0/q\BA_0 \simeq \QQ$. We define

\begin{align*}
\SL(\infty) &= \{ x \in \BU_q^- \mid (x,x) \in \BA_0\}, \\
\SL_{\ZZ}(\infty) &= \{ x \in {}_{\BA}\BU_q^- \mid  (x, x) \in \BA_0\}. 
\end{align*}  
Then $\SL(\infty)$ is an $\BA_0$-module, and $\SL_{\ZZ}(\infty)$
is a module over $\BA_0 \cap \BA = \ZZ[q]$.
It is shown that if the canonical basis $\bB$ exists, then $\bB$
gives an $\BA_0$-basis of $\SL(\infty)$, and
a $\ZZ[q]$-basis of $\SL_{\ZZ}(\infty)$. 
\par
We define a subset $\wt\CB$ of $\BU_q^-$ by
\begin{equation*}
\tag{1.7.1}  
\wt\CB = \{ x \in \BU_q^- \mid \ol x = x, (x,x) \in 1 + q\BA_0\}.
\end{equation*}  
Clearly, if $x \in \wt\CB$, then $-x \in \wt\CB$.
If there exists a basis $\CB$ of $\BU_q^-$ such that $\wt\CB = \CB \sqcup -\CB$,
then $\wt\CB$ is called the canonical signed basis of $\BU_q^-$.
Note that the choice of $\CB$ such that $\wt\CB = \CB \sqcup -\CB$ is not unique. 

\remarks{1.8.}
(i) \
In the case where $\BU_q^-$ is symmetric, the canonical basis $\bB$ was first
constructed in \cite{L-can}, by making use of the representation theory of quivers
(but the geometry of quivers is not used). 
In the general case, it was proved in \cite[14.2]{L-book} that $\wt\CB$ is the canonical
signed basis, by using the geometric categorification of $\BU_q^-$ in terms of quivers
(see Introduction). 
This means that the canonical basis $\bB$ for $\BU_q^-$ is determined, up to $\pm 1$.
After that in \cite[19.2]{L-book} he could normalize the sign to obtain the canonical basis $\bB$,
by the aide of Kashiwara's theory of crystals. 
\par
(ii) \  In this paper, we prove that $\wt\CB$ is the canonical signed basis,
by making use of the folding theory of quantum groups.  By the folding theory,
the construction of the canonical bases is reduced to the $B_2$-case.
In this case, we construct it by using the piecewise linear bijections. 
\par\bigskip

\section{ Foldings of quantum groups }

\para{2.1.}
Let $X = (I, (\ ,\ ))$ be a Cartan datum of finite type as in 1.1.
We assume that $X$ is symmetric.
A permutation $\s : I \to I$ is called an admissible automorphism on $X$
if it satisfies the property that
$(\a_i, \a_{i'}) = (\a_{\s(i)}, \a_{\s(i')})$ for any $i,i' \in I$,
and that $(\a_i, \a_{i'}) = 0$ if $i$ and $i'$ belong to the same $\s$-orbit in $I$.
Assume that $\s$ is admissible. 
We denote by $\Bn$ the order of $\s: I \to I$. 
Let $J$ be the set of
$\s$-orbits in $I$. For each $j \in J$, set $\a_j = \sum_{i \in J}\a_i$,
and consider the subspace $\bigoplus_{j \in J}\QQ \a_j$ of $E$ with basis $\a_j$.
We denote by $|j|$ the size of of the orbit $j$ in $I$.
The restriction of the form $(\ ,\ )$ on $\bigoplus_{j \in J}\QQ\a_j$ is given by

\begin{equation*}
(\a_j, \a_{j'}) = \begin{cases}
    (\a_i,\a_i)|j|, \quad (i \in j) &\quad\text{ if } j = j', \\
    \sum_{i \in j, i' \in j'}(\a_i, \a_{i'})   &\quad\text{ if } j \ne j'. 
                  \end{cases}
\end{equation*}
Then $\ul X = (J, (\ ,\ ))$ is a Cartan datum, which is called the Cartan datum
induced from $(X, \s)$.
Now  $\s$ acts naturally on the root lattice $Q$,
and the set $Q^{\s}$ of $\s$-fixed elements in $Q$
is identified with the root lattice of $\ul X$.  We have $Q^{\s}_+ = \sum_{j \in J}\NN \a_j$.  
\par
The admissible automorphism $\s : I \to I$  induces 
an algebra automorphism on $\BU_q^-$ by $f_i \mapsto f_{\s(i)}$ for each $i \in I$,
which we also denote by $\s$.  The action of $\s$ leaves ${}_{\BA}\BU_q^-$ invariant, 
and we denote by ${}_{\BA}\BU_q^{-,\s}$ the fixed point subalgebra of ${}_{\BA}\BU_q^-$. 
\par
Let $\ul{\BU}_q^- = \BU_q^-(\ul X)$
be the negative half of the quantum group $\ul\BU_q$ associated to $\ul X$.
Thus $\ul\BU_q^-$ is an associative $\QQ(q)$-algebra, with generators $f_j \, (j \in J)$
and $q$-Serre relations as in (1.2.1).  The integral form ${}_{\BA}\ul\BU_q^-$ of $\ul\BU_q^-$
is the $\BA$-subalgebra generated by $f_j^{(n)}$ \ ($j \in J, n \in \NN$). 

\para{2.2.}
Let $\s: I \to I$ be as in 2.1. 
In this section, we assume that $X$ is irreducible.
Then $\Bn = 2$ or 3, which is a prime number $p$. Let $\FF = \ZZ/p\ZZ$
be the finite field of $p$ elements, and set $\BA' = \FF[q,q\iv] = \BA/p \BA$.
We consider the $\BA'$-algebra

\begin{equation*}
  {}_{\BA'}\BU_q^{-,\s} = \BA' \otimes_{\BA}{}_{\BA}\BU_q^{-,\s}
          \simeq {}_{\BA}\BU_q^{-,\s}/p({}_{\BA}\BU_q^{-,\s}).
\end{equation*}

Let $\BJ$ be an $\BA'$-submodule of $_{\BA'}\BU_q^{-,\s}$ generated by
$\sum_{0 \le i < k}\s^i(x)$, where $k$ is the smallest integer $\ge 1$
such that $\s^k(x) = x$.
Then $\BJ$ is the two-sided ideal of
$_{\BA'}\BU_q^{-\s}$.  We define an $\BA'$-algebra $\BV_q$ as the quotient algebra
of ${}_{\BA'}\BU_q^{-,\s}$,

\begin{equation*}
\tag{2.2.1}  
\BV_q = {}_{\BA'}\BU_q^{-,\s}/\BJ.
\end{equation*}
Let $\pi : {}_{\BA'}\BU_q^{-,\s} \to \BV_q$ be the natural projection. 

\par
For each $j \in J$, and $a \in \NN$, set $\wt f^{(a)}_j = \prod_{i \in j}f_i^{(a)}$.
Since $\s$ is admissible, $f_i^{(a)}$ and $f_{i'}^{(a)}$ commute each other for $i,i' \in j$.
Hence $\wt f_j^{(a)}$ does not depend on the order of the product, and we have
$\wt f_j^{(a)} \in {}_{\BA}\BU_q^{-,\s}$.  We denote by the same symbol the image of
$\wt f_j^{(a)}$ on ${}_{\BA'}\BU_q^{-,\s}$. Thus $\pi(\wt f_j^{(a)}) \in \BV_q$ can be defined. 
In the case where $a = 1$, we set $\wt f_j^{(1)} = \wt f_j = \prod_{i \in j}f_i$.
\par
We define an $\BA'$-algebra ${}_{\BA'}\ul\BU_q^-$ by
${}_{\BA'}\ul\BU_q^- = \BA'\otimes_{\BA}{}_{\BA}\ul\BU_q^-$. 

\begin{thm}[{[SZ1]}] 
The assignment $f_j^{(a)} \to \pi(\wt f_j^{(a)})$ \ ($j \in J$)
gives an $\BA'$-algebra isomorphism $\Phi : {}_{\BA'}\ul\BU_q^- \isom \BV_q$.
\end{thm}  

\par\noindent
\remark{2.4.}
The statement as in Theorem 2.3 was proved by \cite[Thm. 0.4]{SZ1} in the case of
finite type, and was generalized in \cite{SZ2} to the affine case, by using
the PBW-bases of $\BU_q^-$. Finally, it was proved in \cite{MSZ1} for the
case of Kac-Moody type, in general, assuming the existence of the canonical basis
in $\BU_q^-$. The proof in \cite{SZ1},\cite{SZ2} is completely elementary.

\para{2.5.}
Let $W$ (resp. $\ul W$) be the Weyl group generated by $\{ s_i \mid i \in I\}$ 
(resp. $\{ s_j \mid j \in J\}$) associated to $X$ (resp. $\ul{X}$).
Let $\s : W \to W$ be the automorphism of $W$ defined by $s_i \mapsto s_{\s(i)}$, and
$W^{\s}$ the fixed point subgroup of $W$ by $\s$.
For $j \in J$, set $w_j = \prod_{i \in j}s_i$. Since $\s$ is admissible, $w_j$
is independent from the order of the product, and so $w_j \in W^{\s}$. 
It is known that the correspondence $s_j \mapsto w_j$ gives an isomorphism of groups
$\ul W \isom W^{\s}$.
Let $w_0$ be the longest element of $W$ with $l(w_0) = N$, and $\ul{w}_0$ the longest
element of $\ul{W}$ with $l(\ul{w}_0) = \ul N$. 
Let $\ul{w_0} = s_{j_1}\cdots s_{j_{\ul N}}$ be a reduced expression of $\ul{w}_0$, then
we have $w_0 = w_{j_1}\cdots w_{j_{\ul N}}$ and $\sum_{1 \le k \le \ul N}l(w_{j_k}) = N$,
where $w_{j_k} = \prod_{i \in j_k}s_i$. 
It follows that
\begin{equation*}
\tag{2.5.1}  
w_0 = \biggl(\prod_{k_1 \in j_1}s_{k_1}\biggr)\cdots
          \biggl(\prod_{k_{\ul N} \in j_{\ul N}}s_{k_{\ul N}}\biggr) = s_{i_1}\cdots s_{i_N}
\end{equation*}
gives a reduced expression of $w_0$.
For a reduced sequence $\ul \Bh = (j_1, \dots, j_{\ul N})$ of $\ul{w}_0$,
we define a reduced sequence $\Bh = (i_1, \dots, i_N)$ of $w_0$ by (2.5.1), namely,
\begin{equation*}
\tag{2.5.2}  
  \Bh = (\underbrace{i_1, \dots, i_{|j_1|}}_{w_{j_1}},
         \underbrace{i_{|j_1|+1}, \dots, i_{|j_1|+|j_2|}}_{w_{j_2}},
            \dots,
         \underbrace{i_{|j_1| + \cdots + |j_{\ul N-1}|+1}, \dots, i_N}_{w_{j_{\ul N}}}).
\end{equation*}  
(Although the expression of $w_j$ is not unique, we ignore the difference of
mutually commuting factors.)

\para{2.6.}
Take reduced sequences $\Bh$ of $w_0$ and $\ul\Bh$ of $\ul{w}_0$.  
Let $\ul\vD^+$ be the set of positive roots associated to $\ul X$.
Then $\ul\Bh$ gives a total order of $\ul\vD^+$ as
$\ul\vD^+ = \{ \ul\b_1, \dots, \ul\b_{\ul N} \}$.   Let
\begin{align*}
\SX_{\Bh} &= \{ L(\Bc, \Bh) \mid \Bc \in \NN^N\}, \\
\SX_{\ul\Bh} &= \{ L(\ul\Bc, \ul\Bh) \mid \ul\Bc \in \NN^{\ul N} \}
\end{align*}
be the PBW-basis of $\BU_q^-$
associated to $\Bh$, and the PBW-basis of $\ul\BU_q^-$ associated to $\ul\Bh$.
Here we have

\begin{align*}
L(\Bc,\Bh) &= f_{\b_1}^{(c_1)}f_{\b_2}^{(c_2)}\cdots f_{\b_N}^{(c_N)}, \\
L(\ul\Bc, \ul\Bh) &= f_{\ul\b_1}^{(\ul c_1)}f_{\ul\b_2}^{(\ul c_2)}
\cdots f_{\ul\b_{\ul N}}^{(\ul c_{\ul N})},
\end{align*}
where $f_{\ul\b_k}^{(\ul c_k)} = T_{j_1}\cdots T_{j_{k-1}}(f_{j_k}^{(\ul c_k)})$,
and $T_j$ ($j \in J)$ is the braid group action on $\ul\BU_q$. 

Now choose $\Bh, \ul\Bh$ as in 2.5.  Following the expression in (2.5.2),
we write $\Bc = (c_1, \dots, c_N) \in \NN^N$ as $\Bc = (\Bd^{(1)}, \dots, \Bd^{(\ul N)})$,
where $\Bd^{(k)} \in \NN^{|j_k|}$, namely,
\begin{equation*}
  \Bc = (\underbrace{c_1, \dots, c_{|j_1|}}_{\Bd^{(1)}},
          \underbrace{c_{|j_1|+1}, \dots, c_{|j_1| + |j_2|}}_{\Bd^{(2)}}, ,
          \dots,
          \underbrace{c_{|j_1|+ \cdots |j_{\ul N-1}|+1}, \dots, c_N}_{\Bd^{(\ul N)}}).
\end{equation*}  
Corresponding to the action of $\s$ on the orbit
$j_k \subset I$,
we define an action of $\s$ on $\Bd^{(k)}$ as a permutation of coefficients,
thus the action of $\s$ on $\NN^N$
can be defined. Let $\NN^{N,\s}$ be the set of $\s$-fixed elements in $\NN^N$.
For $\ul\Bc = (\ul c_1, \dots, \ul c_{\ul N}) \in \NN^{\ul N}$, define
$\Bc = (\Bd^{(1)}, \dots, \Bd^{(\ul N)}) \in \NN^N$ by the condition that
$\Bd^{(k)} = (\ul c_k, \dots, \ul c_k) \in \NN^{|j_k|}$ for each $k$.  Then
$\Bc \in \NN^{N,\s}$, and the assignment $\ul\Bc \mapsto \Bc$ gives a bijection
$\NN^{\ul N} \isom \NN^{N,\s}$. 

\par
The following result was proved in \cite[Thm. 1.14]{SZ1} in the course of the proof
of Theorem 0.4 in \cite{SZ1}.

\begin{prop}  
Under the setup in 2.6, we have
\begin{enumerate}
\item \ $\s$ acts on $\SX_{\Bh}$ as a permutation,
$\s : L(\Bc,\Bh) \mapsto L(\s(\Bc), \Bh)$.
$L(\Bc, \Bh)$ is $\s$-invariant if and only if $\Bc \in \NN^{N,\s}$.
\item \ Under the bijection $\NN^{\ul N} \isom \NN^{N,\s}, \ul\Bc \mapsto \Bc$,
the assignment $L(\ul\Bc, \ul\Bh) \to L(\Bc, \Bh)$ gives a bijection
\begin{equation*}
\SX_{\ul\Bh} \isom \SX_{\Bh}^{\s},
\end{equation*}
where $\SX_{\Bh}^{\s}$ is the set of $\s$-fixed PBW bases in $\SX_{\Bh}$. 
\item \ The bijection in (ii) is compatible with the isomorphism $\Phi$ in
Theorem 2.3, namely on $\BV_q$, we have
\begin{equation*}  
\tag{2.7.1}
\Phi(L(\ul\Bc, \ul\Bh)) = \pi(L(\Bc, \Bh)).
\end{equation*}
In particular, $\SX_{\ul\Bh}$ gives an $\BA$-basis of ${}_{\BA}\ul\BU_q^-$.
\item \ By (iii), the basis $\bB_{\ul\Bh}$ exists.
Similar statements as in (i) $\sim$ (iii) also hold for 
the basis $\bB_{\Bh}$ on $\BU_q^-$ and $\bB_{\ul\Bh}$ on $\ul\BU_q^-$.  
In particular, $\s$ acts on $\bB_{\Bh}$ as a permutation, and let $\bB_{\Bh}^{\s}$
be the set of $\s$-fixed elements. Then we have 
a bijection $\bB_{\ul\Bh} \simeq \bB_{\Bh}^{\s}$,
$b(\ul\Bc, \ul\Bh) \mapsto b(\Bc, \Bh)$ as in (ii), and we have 
\begin{equation*}
\Phi(b(\ul\Bc, \ul\Bh)) = \pi(b(\Bc, \Bh)).
\end{equation*}
\end{enumerate}  
\end{prop}

\remarks{2.8.}
(i) \
The third assertion of Proposition 1.4 for $\ul\BU_q^-$, namely that 
$\SX_{\ul\Bh}$ gives an $\BA$-basis of ${}_{\BA}\ul\BU_q^-$,
is a direct consequence of the formula (2.7.1), together with the corresponding property
for $\BU_q^-$. In fact, this was
proved in \cite{SZ1}, just by assuming the property
$L(\ul\Bc, \ul\Bh) \in {}_{\BA}\ul\BU_q^-$, which is verified
by a direct computation for the case $B_2$ or $G_2$. The argument here is much simpler than
the computation mentioned in Remark 1.5.
\par
(ii) \ The first assertion of Proposition 1.4 for $\ul\BU_q^-$, namely,
the $\ZZ$-submodule of $\ul\BU_q^-$ spanned by $\SX_{\ul\Bh}$ is independent from
$\ul\Bh$, is also proved by using 
the corresponding property for $\BU_q^-$, and by (2.7.1).

\para{2.9.}
The following result was proved in Theorem 4.18 and Remark 4.19 in \cite{MSZ1}.
Note that \cite{MSZ1} discusses the case where $\ul\BU_q^-$ is of Kac-Moody type,
under the assumption that the canonical basis $\bB$ for $\BU_q^-$ 
satisfies certain good properties.  This condition is satisfied for the case
where $\BU_q^-$ is of finite type,  thus the results in \cite{MSZ1}
can be applied to the case of finite type. Note that the first assertion of the
following theorem was first proved in \cite[Prop. 1.23]{SZ1}, which
is immediate from Proposition 2.7.

\begin{thm}  
Assume that $\BU_q^-$ is irreducible of finite type.
Let $\bB_{\ul\Bh}$ be the basis of $\ul\BU_q^-$ given in Proposition 2.7. 
\begin{enumerate}
\item \ Assume that $p = 3$, namely, $\ul\BU_q^-$ is of type $G_2$.
Then $\ul\bB = \bB_{\ul \Bh}$ is independent of $\ul\Bh$, and 
gives the canonical basis of $\ul\BU_q^-$.  
\item \ Assume that $p = 2$.  Then $\wt\CB = \bB_{\ul \Bh} \sqcup -\bB_{\ul\Bh}$
gives the canonical signed basis of $\ul\BU_q^-$.
In particular, up to sign, $\bB_{\ul\Bh}$ is independent of the choice of $\ul\Bh$.
\end{enumerate}  
\end{thm}  

\para{2.11.}
Assume that $\wt\CB = \CB \sqcup -\CB$ is the canonical signed basis of $\ul\BU_q^-$,
where $\CB$ is a basis of $\ul\BU_q^-$. 
For each $x \in \ul\BU_q^-$ and $j \in J$, let $\ve_j(x)$ be the largest integer $n$
such that $x \in f_j^{(n)}\ul\BU_q^-$. 
For $j \in J$ and $n \in \NN$, set $\CB_{j,n} = \{ b \in \CB \mid \ve_j(b) = n\}$. 
Then we have a partition $\CB = \bigsqcup_{n \ge 0}\CB_{j,n}$.
\par
Note that the following results were proved in \cite{L-book} by using the geometry
of quivers. But they were reproved in \cite{MSZ1} for $\CB = \bB_{\ul\Bh}$
in an elementary way, in the course of the proof of Theorem 2.10 (ii).
\par\medskip\noindent
(2.11.1) \ We have $\bigcap_{j \in J}\CB_{j,0} = \{ 1\}$. 
\par\medskip\noindent
(2.11.2) \ Take $b \in \CB_{j,0}$.
Then for any $a > 0$, there exists a unique element $b' \in \CB_{j,a}$ such that
\begin{equation*}
  \pm b' \equiv f_j^{(a)}b  \mod f_j^{a+1}\ul\BU_q^-.
\end{equation*}  
The correspondence $b \mapsto b'$ gives a bijection $\pi_{j,a}: \CB_{j,0} \isom \CB_{j,a}$. 
\par\medskip
We define a Kashiwara operator $ F_j : \CB_{j,n} \to \CB_{j,n+1}$
by
\begin{equation*}
\tag{2.11.3}
\begin{CD}
  F_j : \CB_{j,n} @>\pi_{j,n}\iv >>  \CB_{j,0}
                @>\pi_{j,n+1} >>  \CB_{j, n+1}.    
\end{CD}    
\end{equation*}
Then $F_j$ gives a bijection $\CB_{j,n} \isom \CB_{j, n+1}$.
\par\medskip
Assume that $\ul\Bh = (j_1, \dots, j_{\ul N})$ with $j = j_1$.
In this case, for $x = L(\ul\Bc, \ul\Bh)$ with $\ul\Bc = (c_1, \dots, c_{\ul N})$,
we have $\ve_j(x) = c_1$. The following result is easily verified from this.
\par\medskip\noindent
(2.11.4) \ Assume that $j = j_1$, then we have
$F_j(b(\ul\Bc, \ul\Bh)) = b(\ul\Bc^+, \ul\Bh)$, 
where we set $\ul\Bc^+ = (c_1+1, c_2, \dots, c_{\ul N})$ for
$\ul\Bc = (c_1, \dots, c_{\ul N})$.
  
\remark{2.12.}
In view of Theorem 2.10, in order to construct
the canonical basis for $\ul\BU_q^-$ in an elementary way,  
it is enough to show the following.
\par\medskip\noindent
(2.12.1) \ Assume that $\ul\BU_q^-$ is type $B_2$.
Then $\bB_{\ul\Bh}$ is independent from the choice of $\ul\Bh$.
\par\medskip
We prove (2.12.1) in next section.

\par\bigskip
\section{ Piecewise linear bijections}

\para{3.1.}
In this section, we prove (2.12.1) by making use of the piecewise linear
bijections associated to quantum groups, introduced by Lusztig \cite{L-piece}.
We start with the explanation of piecewise linear bijections.
First consider the case where $\BU_q^-$ is of type $A_2$ with
$I = \{ 1,2\}$.  Then we have two reduced sequences $\Bh = (1,2,1)$ and $\Bh' = (2,1,2)$
of $w_0 \in S_3$.
Accordingly, $\bB_{\Bh} = \{ b(\Bc,\Bh) \mid \Bc \in \NN^3\}$,
$\bB_{\Bh'} = \{ b(\Bc', \Bh') \mid \Bc' \in \NN^3\}$.  Since
$\bB_{\Bh} = \bB_{\Bh'}$ in $\BU_q^-$,
there exists a bijection $\vf : \NN^3 \to \NN^3$ such that
$b(\Bc, \Bh) = b(\vf(\Bc), \Bh')$. The map
$\vf : \Bc = (c_1, c_2, c_3) \mapsto \Bc' = (c_1', c_2', c_3')$ is given by the formula
\begin{equation*}
  c_1' = c_2 + c_3 - \min(c_1, c_3), \quad c_2' = \min (c_1, c_3),
  \quad c_3' = c_1 + c_2 - \min (c_1,c_3).
\end{equation*}  
It is easy to see that $\vf^2 = \id$.
The map $\vf$ is called the piecewise linear bijection. 
In \cite{L-piece}, Lusztig described this map by using the notion of semifields.
We consider a bijection $\io : \ZZ \isom K$, where the operations
$\min(a,b), a+b, a-b$ in $\ZZ$ corresponds to the operations $a+b, ab, a/b$ in $K$.
Then $K$ has a structure of a semifield; the multiplication is commutative,
and the distribution law holds, furthermore, the subset $\io(\NN)$ in $K$ is
closed under the operations $a+b,ab, a/(a+b)$.
Using this notation, the piecewise linear bijection $\vf$ is written as
\begin{equation*}
\tag{3.1.1}  
\vf : \NN^3 \to \NN^3, \quad (a,b,c) \mapsto (bc/(a+c), a+c, ab/(a+c)).
\end{equation*}  

\para{3.2.}
Assume that $X$ is a Cartan datum of type $ADE$.  Let $\Bh = (i_1, \dots, i_N)$
be a reduced sequence of $w_0 \in W$, and consider the canonical basis $\bB = \bB_{\Bh}$ and
the PBW basis $\SX_{\Bh}$ associated to $\Bh$. For a given $\Bc = (c_1, \dots, c_N) \in \NN^N$,
we denote the canonical basis $b = b(\Bc, \Bh) \in \bB_{\Bh}$
by $i_1^{c_1}i_2^{c_2}\cdots i_N^{c_N}$. 
For another reduced sequence $\Bh' = (i_1', \dots, i_N')$, $b$ is written as
$b = b(\Bc', \Bh')$, $\Bc' = (c_1', \dots, c_N')$, and the relation 
$i_1^{c_1}\cdots i_N^{c_N} = {i'_1}^{c_1'} \cdots {i'_N}^{c'_N}$ gives two different labelings of
$\bB = \bB_{\Bh} = \bB_{\Bh'}$.
Here $\Bh'$ is obtained from $\Bh$ by applying repeatedly two types of operations,
$(i,j,i) \mapsto (j,i,j)$ (with $(\a_i, \a_j) = -1$) or $(i,j) \mapsto (j,i)$
(with $(\a_i, \a_j) = 0$).
\par\medskip
(a) \ First assume that $\Bh'$ is obtained from $\Bh$ by
replacing $(i_k, i_{k+1}, i_{k+2}) = (i,j,i)$ by $(i'_k, i'_{k+1}, i'_{k+2}) = (j,i,j)$.
If $b(\Bc, \Bh) = b(\Bc',\Bh')$, then by the discussion in $\BU_q(\Fs\Fl_3)$,
$\Bc'$ is obtained from $\Bc$ by replacing $\Bx = (c_k, c_{k+1}, c_{k+2})$ by
$\Bx' = (c_k', c'_{k+1}, c'_{k+2})$, and leaving other parts invariant, where
$\Bx \mapsto \Bx'$ is given by the formula in (3.1.1). 
\par\smallskip
(b) \ 
Next assume that $\Bh'$ is obtained from $\Bh$ by replacing
$(i_k, i_{k+1}) = (i,j)$ by $(i'_k, i'_{k+1}) = (j,i)$. Then $\Bc'$ is obtained
by replacing $\Bx = (c_k, c_{k+1})$ by $\Bx' = (c_{k+1}, c_k)$, and leaving other parts
invariant.
\par\smallskip
Thus, in general, if $b(\Bc, \Bh) = b(\Bc', \Bh')$, then 
$\Bc'$ is obtained from $\Bc$ by repeating the above operations (a) or (b). 

\para{3.3.}
Assume that $X$ is a Cartan datum of type $A_3$ with $I = \{ 1,2, 1'\}$.
Let $\s : I \to I$ be the admissible automorphism defined by
$\s : 1 \lra 1', 2\lra 2$. Then the Cartan datum $\ul X$ induced from $(X, \s)$
is of type $B_2$ with $J = \{ \ul 1, \ul 2\}$, where $\ul 1 = \{ 1,1'\}$ and $\ul 2 = \{ 2\}$. 
Let $\ul\Bh = (\ul 1, \ul 2, \ul 1, \ul 2), \ul\Bh' = (\ul 2, \ul 1, \ul 2, \ul 1)$
be two reduced sequences of $\ul w_0$ of $\ul W$, where $\ul W$ is the Weyl group of type $B_2$.
Correspondingly, we have two reduced sequences $\Bh, \Bh'$ of $w_0 \in W$ as given in (2.5.2),
where $W $ is the Weyl group of type $A_3$,
\begin{equation*}
\Bh = (1,1',2,1',1,2), \qquad \Bh' = (2, 1,1',2,1',1).
\end{equation*}  
Note that the choice of $\Bh, \Bh'$ is not unique (the order of $1,1'$ can be reversed),
but here we fix them in this way.
\par
Let $\bB^{\s}$ be the set of $\s$-fixed elements of $\bB$ (see Proposition 2.7). 
Take $b = b(\Bc, \Bh) = b(\Bc',\Bh')$, and assume that $b \in \bB^{\s}$.
Then $\Bc$ is written as $\Bc = (a,a,b,c,c,d)$, and $\Bc'$ is written as
$\Bc' = (a', b',b', c',d',d')$. The piecewise linear bijection $\Bc \mapsto \Bc'$
is described by using the procedure as given in 3.2.  This map was explicitly computed
in \cite{L-piece}, through the operation
$1^a1'^a2^b1'^c1^c2^d \mapsto 2^{a'}1^{b'}1'^{b'}2^{c'}1'^{d'}1^{d'}$, as follows.

\begin{equation*}
\tag{3.3.1}
\begin{aligned}
  (1) \qquad   &1^a(1'^a2^b1'^c)1^c2^d  \\
  (2) \qquad    &1^a2^{\frac{bc}{a+c}}1'^{a+c}(2^{\frac{ab}{a+c}}1^c2^d) \\
  (3) \qquad   &1^a12^{\frac{bc}{a+c}}(1'^{a+c}1^{\frac{cd(a+c)}{\a}})
              2^{\frac{\a}{a+c}}1^{\frac{abc}{\a}} \\
  (4) \qquad    &(1^a2^{\frac{bc}{a+c}}1^{\frac{cd(a+c)}{\a}})1'^{a+c}
              2^{\frac{\a}{a+c}}1^{\frac{abc}{\a}} \\
  (5) \qquad    &2^{\frac{bc^2d}{\ve}}1^{\frac{\ve}{\a}}(2^{\frac{abc\a}{(a+c)\ve}}
                     1'^{a+c}2^{\frac{\a}{a+c}})1^{\frac{abc}{\a}} \\
  (6) \qquad   &2^{\frac{bc^2d}{\ve}}1^{\frac{\ve}{\a}}1'^{\frac{\ve}{\a}}2^{\frac{\a^2}{\ve}}
                     1'^{\frac{abc}{\a}}1^{\frac{abc}{\a}}, \\
\end{aligned}
\end{equation*}
where $\a = ab + ad + cd, 
   \ve = a^2b + a^2d + c^2d + 2acd = a^2b + a^2d + c^2d + acd$.  
Here each step corresponds to the operation (a) or (b) in 3.2.
\par
Now the piecewise linear bijection $\vf : \NN^6 \to \NN^6$ induces
a piecewise linear bijection $\vf : \NN^{6,\s} \isom \NN^{6,\s}$. By (3.3.1),
this map is given by $\Bc = (a, a, b,c, c,d) \mapsto \Bc' = (a',b',b',c',d',d')$, where
\begin{equation*}
\tag{3.3.2}  
a' = bc^2d/\ve, \qquad b' = \ve/\a, \qquad c' = \a^2/\ve, \qquad d' = abc/\a.
\end{equation*}
Hence we have a bijection $\vf : \NN^4 \to \NN^4$, given by
$(a,b,c,d) \mapsto (a',b',c',d')$.
\para{3.4.}
Let $\SX_{\ul\Bh}$ be the PBW basis of $\ul\BU_q^-$ 
and $\bB_{\ul\Bh}$ the basis of $\ul\BU_q^-$ associated to $\ul\Bh$
as given in Proposition 2.7.
We define $\bB_{\ul\Bh'}$ and $\SX_{\ul\Bh'}$ similarly. 
By Proposition 2.7, we have $\SX_{\ul\Bh} \simeq \SX_{\Bh}^{\s}$, and
this bijection is obtained by the bijection
$\NN^4 \isom \NN^{6,\s}, \ul\Bc = (a,b,c,d) \mapsto \Bc = (a,a,b,c,c,d)$.
Similarly, the bijection $\SX_{\ul\Bh'} \simeq \SX_{\Bh'}^{\s}$ is obtained
by the bijection $\NN^4 \simeq \NN^{6, \s}$,
$\ul\Bc' = (a',b',c',d') \mapsto \Bc' = (a',b',b',c',d',d')$.
\par
We consider the bijection $\vf : \NN^4 \isom \NN^4$ given in (3.3.2).
Then under the identification $\NN^4 \simeq \NN^{6,\s}$ (in two ways),
$\vf$ induces a piecewise linear bijection $\NN^{6,\s} \isom \NN^{6,\s}$.
This is nothing but the bijection $\Bc \mapsto \Bc'$ for
$b(\Bc, \Bh) = b(\Bc',\Bh') \in \bB^{\s}$ discussed in 3.3.
By Proposition 2.7 (iv), we have $\Phi(b(\ul\Bc, \ul\Bh)) = \pi(b(\Bc,\Bh))$,
and similarly, $\Phi(b(\ul\Bc, \ul\Bh')) = \pi(b(\Bc',\Bh'))$. 
It follows that $\Phi(b(\ul\Bc,\ul\Bh)) = \Phi(b(\ul\Bc', \ul\Bh')$.
Since $\Phi$ is an isomorphism on the algebra over $\BA' = (\ZZ/2\ZZ)[q,q\iv]$,
we have the following.

\begin{prop}  
The relationship between $\bB_{\ul\Bh}$ and $\bB_{\ul\Bh'}$ is given by
$b(\ul\Bc,\ul\Bh) = \pm b(\ul\Bc', \ul\Bh')$, where $\ul\Bc \mapsto \ul\Bc'$
is the piecewise linear bijection $\vf : \NN^4 \isom \NN^4$ defined in (3.3.2).
\end{prop}

In order to show $\bB_{\ul\Bh} = \bB_{\ul\Bh'}$, it is enough to see that
$b(\ul\Bc, \ul\Bh) = b(\ul\Bc', \ul\Bh')$ for $\ul\Bc' = \vf(\ul\Bc)$. 
The following lemma is easily verified from the definition of $\vf$ in (3.1.1). 
(Note that here we use the standard notation for $\NN$, not the semifield notation.) 

\begin{lem}  
Let $\vf : \NN^3 \isom \NN^3$ be the piecewise linear bijection defined in (3.1.1).
Assume that $\vf(x,y,z) = (x',y',z')$.  Then we have

\begin{equation*}
  \vf\iv(x'+1,y',z') = \begin{cases}
                           (x, y, z+1)  &\quad\text{ if } x \le z, \\
                           (x-1, y-1, z)  &\quad\text{ if } x > z.
                      \end{cases}
\end{equation*}
\end{lem}

The following lemma is a generalization of Lemma 3.6 (for $A_2$ case)
to the case $B_2$. 

\begin{lem}  
Let $\vf : \NN^4 \to \NN^4, (a,b,c,d) \mapsto (a',b',c',d')$ be the piecewise
bijection defined in (3.3.2).  Then $\vf\iv(a'+1, b', c',d')$ is given as follows;

\begin{enumerate}
\item \  $(a-1, b+1, c,d)$  \quad  if  $a > cd(a+c)/\a, \ a> c$,  
\item \  $(a, b-1, c+1, d)$  \quad if $a \le cd(a+c)/\a, \ ab/(a+c)> d, \ a\le c$,  
\item \ $(a,b, c, d+1)$   \hspace{3mm}\qquad if $a \le cd(a+c)/\a, \ ab/(a+c) \le d$,
\end{enumerate}
where $\a = ab + ad + cd$. Note that the condition in the right hand side
is written by using the semifield notation, 
but the inequalities are on $\ZZ$. 
\end{lem}
\begin{proof}
We give an outline of the proof. Let $\Bc = (a,a, b, c,c,d) \in \NN^6$
and set $\Bc' = \vf(\Bc) = (a',b',b',c',d',d') \in \NN^6$, where
$\vf$ is defined as in (3.3.2).
We compute $\vf\iv(a'+1, b', b',c',d',d')$. Consider the operations in (3.3.1)
in a reverse order from (6) to (1), by replacing $2^{\frac{bc^2d}{\ve}}$  by
$2^{\frac{bc^2d1}{\ve}}$ (here $bc^2d1$ is in the semifield notation corresponds to
$a'+1$ for $a' = \frac{bc^2d}{\ve}$).  Then
\begin{align*}
2^{\frac{bc^2d1}{\ve}}1^{\frac{\ve}{\a}}(1'^{\frac{\ve}{\a}}2^{\frac{\a^2}{\ve}}
                     1'^{\frac{abc}{\a}})1^{\frac{abc}{\a}} 
 &= (2^{\frac{bc^2d1}{\ve}}1^{\frac{\ve}{\a}}2^{\frac{abc\a}{(a+c)\ve}})
                      1'^{a+c}2^{\frac{\a}{a+c}}1^{\frac{abc}{\a}} \\
 &= \begin{cases}
 (1^a2^{\frac{bc}{a+c}}1^{\frac{cd(a+c)1}{\a}})1'^{a+c}2^{\frac{\a}{a+c}}1^{\frac{abc}{\a}} 
    &\quad\text{ if } a \le cd(a+c)/\a, \\
 (1^{a/1}2^{\frac{bc1}{a+c}}1^{\frac{cd(a+c)}{\a}})1'^{a+c}2^{\frac{\a}{a+c}}1^{\frac{abc}{\a}} 
    &\quad\text{ if } a > cd(a+c)/\a.
   \end{cases}
\end{align*}
Th last equality follows from Lemma 3.6.
This corresponds to the operations $(6) \Rightarrow (5)$ and
$(5) \Rightarrow (4)$.
In the last formula, we consider the two cases, separately.
\par\medskip
(A) \ Assume that $a \le cd(a+c)/\a$.  Then
\begin{align*}
1^a2^{\frac{bc}{a+c}}(1^{\frac{cd(a+c)1}{\a}}1'^{a+c})
         2^{\frac{\a}{a+c}}1^{\frac{abc}{\a}} 
&= 1^a2^{\frac{bc}{a+c}}1'^{a+c}(1^{\frac{cd(a+c)1}{\a}}2^{\frac{\a}{a+c}}1^{\frac{abc}{\a}}) \\
&= \begin{cases}
      1^a2^{\frac{bc}{a+c}}1'^{a+c}(2^{\frac{ab}{a+c}}1^c2^{d1}) 
         &\quad\text{ if }     ab(a+c) \le d, \\
      1^a2^{\frac{bc}{a+c}}1'^{a+c}(2^{\frac{ab/1}{a+c}}1^{c1}1^{d}) 
         &\quad\text{ if }     ab(a+c) > d. 
   \end{cases}
\end{align*} 
\par\medskip
(B) \ Assume that $a > cd(a+c)/\a$.
Then 
\begin{align*}
1^{a/1}2^{\frac{bc1}{a+c}}
       (1^{\frac{cd(a+c)}{\a}}1'^{a+c})2^{\frac{\a}{a+c}}1^{\frac{abc}{\a}}  
&= 1^{a/1}2^{\frac{bc1}{a+d}}
       1'^{a+c}(1^{\frac{cd(a+c)}{\a}}2^{\frac{\a}{a+c}}1^{\frac{abc}{\a}})  \\ 
&= 1^{a/1}2^{\frac{bc1}{a+c}}
       1'^{a+c}(2^{\frac{ab}{a+c}}1^{c}2^{d}).  
\end{align*}
These correspond to the operations $(4) \Rightarrow (3)$ and $(3) \Rightarrow (2)$. 
\par
Again, by considering the operation $(2) \Rightarrow (1)$, we obtain various
expressions $1^{a_1}1'^{a_2}2^{a_3}1'^{a_4}1^{a_5'}2^{a_6}$ with certain conditions attached. 
Among them, pick up $\s$-stable elements
$(a_1, \dots, a_6) = (a'', a'', b'', c'', c'', d'') \in \NN^{6,\s}$. Then
the corresponding elements $(a'',b'',c'',d'') \in \NN^4$ with attached conditions
give the required elements (i) $\sim$ (iii) in the lemma. 
\end{proof}

\para{3.8.}
Let $\ul\vD$ be the root system of type $B_2$. By the reduced sequence
$\ul\Bh = (\ul 1, \ul 2, \ul 1, \ul 2)$, 
$\ul\vD^+$ is expressed as $\ul\vD^+ = \{ \ul 1, \ul 1\ul 2, \ul 1\ul 2\ul 2, \ul 2\}$
in this order,
where we use a convention such as $\a_{\ul 1} + \a_{\ul 2} = \ul 1\ul 2$,
$\a_{\ul 1} + 2\a_{\ul 2} = \ul 1\ul 2\ul 2$.
The PBW basis $\SX_{\ul\Bh} = \{ L(\ul\Bc, \ul\Bh) \mid \ul\Bc \in \NN^4\}$
of $\ul\BU_q^-$ is given as
\begin{equation*}
L(\ul\Bc, \ul\Bh)
        = f_{\ul 1}^{(a)}f_{\ul 1\ul 2}^{(b)}f_{\ul 1\ul 2\ul 2}^{(c)}f_{\ul 2}^{(d)}
\end{equation*}
for $\ul\Bc = (a,b,c,d) \in \NN^4$.  Similarly,
for $\ul\Bh' = (\ul 2, \ul 1, \ul 2, \ul 1)$, $\ul\vD^+$ is expressed
as $\ul\vD^+ = \{ \ul 2, \ul{122}, \ul{12}, \ul 1\}$ in this order, and
the PBW basis $\SX_{\ul\Bh'}$ is given as
\begin{equation*}
L(\ul\Bc,\ul \Bh')
      = f_{\ul 2}^{(a)}{f'}_{\ul 1\ul 2\ul 2}^{(b)}{f'}_{\ul 1\ul 2}^{(c)}f_{\ul 1}^{(d)}
\end{equation*}
for $\ul\Bc = (a,b,c,d) \in \NN^4$, where $f_{\ul 1\ul2\ul 2}' = f_{\ul 1\ul 2\ul 2}^*$,
$f'_{\ul 1\ul 2} = f_{\ul 1\ul 2}^*$ (${}^* : \ul\BU_q^- \to \ul\BU_q^-$ is the anti-automorphism
defined by $f_j^* = f_j$ for any $j \in J$).  
\par
The following formulas can be computed, by using various commutation relations
for $f_{\b}^{(n)}$ ($\b \in \ul\vD^+$) (see, for example,  \cite[5.2, 5.3]{L-root}).   

\begin{prop}  
\begin{enumerate}
\item \ The expansion of $f_{\ul 2}L(\ul\Bc, \ul\Bh)$ in terms of $\SX_{\ul\Bh}$
  is given as
\begin{align*}
\tag{3.9.1}    
  f_{\ul 2}L(\ul\Bc, \ul\Bh) &=
    A f_{\ul 1}^{(a-1)}f_{\ul 1\ul 2}^{(b)}f_{\ul 1\ul 2\ul 2}^{(c)}f_{\ul 2}^{(d)}   \\
     &\qquad + B f_{\ul 1}^{(a)}f_{\ul 1\ul 2}^{(b-1)}f_{\ul 1\ul 2\ul 2}^{(c+1)}f_{\ul 2}^{(d)}
             + Cf_{\ul 1}^{(a)}f_{\ul 1\ul 2}^{(b)}f_{\ul 1\ul 2\ul 2}^{(c)}f_{\ul 2}^{(d+1)},
\end{align*}
where
\begin{align*}
  A &= [b + 1] \in q^{-b}(1 + q\ZZ[q]), \\
  B &= q^{2a-b+1}[2][c+1]_{\ul 1} \in q^{2a-b-2c}(1 + q\ZZ[q]), \\
  C &= q^{2a-2c}[d+1] \in q^{2a-2c-d}(1 + q\ZZ[q]).
\end{align*}
\item \ The expansion of $f_{\ul 1}L(\ul\Bc, \ul\Bh')$ in terms of $\SX_{\ul \Bh'}$ is
  given as
\begin{align*}
\tag{3.9.2}    
f_{\ul 1}L(\ul\Bc,\ul\Bh')
  &= Af_{\ul 2}^{(a-2)}{f'}_{\ul 1\ul 2\ul 2}^{(b+1)}{f'}_{\ul 1\ul 2}^{(c)}f_{\ul 1}^{(d)}
        + Bf_{\ul 2}^{(a-1)}{f'}_{\ul 1\ul 2\ul 2}^{(b)}
                 {f'}_{\ul 1\ul 2}^{(c+1)}f_{\ul 1}^{(d)}  \\
    &\qquad + C f_{\ul 2}^{(a)}{f'}_{\ul 1\ul 2\ul 2}^{(b-1)}{f'}_{\ul 1\ul 2}^{(c+2)}
                       f_{\ul 1}^{(d)}
                 + D f_{\ul 2}^{(a)}{f'}_{\ul 1\ul 2\ul 2}^{(b)}{f'}_{\ul 1\ul 2}^{(c)}
                      f_{\ul 1}^{(d + 1)},
\end{align*}
where
\begin{align*}
  A &= [b +1]_{\ul 1} \in q^{-2b}(1 + q\ZZ[q]),  \\
  B &= q^{a - 2b -1}[c+1] \in q^{a -2b -1-c}(1 + q\ZZ[q]),  \\
  C &= q^{2a -2b)}(1-q^2)\begin{bmatrix}
                                  c +2 \\
                                  2
                                \end{bmatrix} \in q^{2a -2b -2c}(1 + q\ZZ[q]), \\
  D &= q^{2a -2c}[d +1]_{\ul 1} \in q^{2a -2c -2d}(1 + q\ZZ[q]).   
\end{align*}  
\end{enumerate}    
\end{prop}  

\para{3.10.}
We express the operation $\vf$ or $\vf\iv$ on $\NN^4$ by $\Bc \mapsto \Bc^{\bullet}$,
and the operation $(a,b,c,d) \mapsto (a +1, b,c,d)$ on $\NN^4$ by
$\Bc \mapsto \Bc^{+}$. 
\par
For $\ul\Bc = (a,b,c,d)$, set $\ul b= b(\ul\Bc, \ul\Bh) \in \bB_{\ul\Bh}$.
By Proposition 3.5, $\ul b$ is written as $\ul b = \pm b(\ul\Bc', \ul\Bh')$ 
for $\ul\Bc' = \vf(\ul\Bc) = (a',b',c',d')$. 
Here we assume that $\ul b = b(\ul\Bc', \ul\Bh')$.
For 
\begin{equation*}
\tag{3.10.1}  
\ul\Bc^{\bullet + } = (a'+1, b', c', d'),
  \qquad \ul\Bc^{\bullet + \bullet} = \vf\iv(\ul\Bc^{\bullet +}), 
\end{equation*}
set $\ul b^{\circ} = b(\ul\Bc^{\bullet + \bullet}, \ul\Bh)$. Then by Proposition 3.5,
one can write as $\ul b^{\circ} = \d b(\ul\Bc^{\bullet +}, \ul\Bh')$ with $\d = \pm 1$.  
\par
Similarly, for $\ul\Bc = (a,b,c,d)$, set $\ul b = b(\ul\Bc, \ul\Bh') \in \bB_{\ul\Bh'}$.
Here we assume that $\ul b = b(\ul\Bc', \ul\Bh)$ with $\ul\Bc' = \vf\iv(\ul\Bc) = (a',b',c',d')$.
For 
\begin{equation*}
{\ul\Bc}^{\bullet +} = (a' +1, b',c',d'), \qquad
  {\ul\Bc}^{\bullet +\bullet} = \vf({\ul \Bc}^{\bullet +}),
\end{equation*}
set ${\ul b}^{\circ} = b({\ul\Bc}^{\bullet + \bullet}, \ul\Bh')$.
Then one can write as ${\ul b}^{\circ} = \d b({\ul\Bc}^{\bullet +}, \ul\Bh)$ with $\d = \pm 1$. 
The following lemma holds.

\begin{lem}  
Consider the expansion
$f_{\ul 2}L(\ul\Bc,\ul\Bh) = \sum_{b_1 \in \bB_{\ul\Bh}}\xi_{b_1}b_1$
$(\xi_{b_1} \in \BA)$ by the basis $\bB_{\ul\Bh}$.  Then we have
\begin{equation*}
\tag{3.11.1}  
\xi_{\ul b^{\circ}} \in q^{-a'}(\d + q\ZZ[q]).
\end{equation*}
\par
A similar formula holds for the expansion 
$f_{\ul 1}L(\ul\Bc, \ul\Bh') = \sum_{b_1 \in \bB_{\ul\Bh'}}\xi_{b_1}b_1$
$(\xi_{b_1} \in \BA)$ by the basis $\bB_{\ul\Bh'}$. 
\end{lem}
\begin{proof}
First we prepare some notations.
Recall that $\SL_{\BZ}(\infty)$ is the $\BZ[q]$-submodule of $\ul\BU_q^-$ spanned by
$\bB_{\ul\Bh}$. By Proposition 1.4, this coincides with the $\ZZ[q]$-submodule
spanned by one of $\bB_{\ul\Bh'}$, $\SX_{\ul\Bh}$, or $\SX_{\ul\Bh'}$.
For $x \in \ul\BU_q^-$, let $\ve_j(x)$ be as in 2.11.
If $j = \ul 1$, $\ve_{\ul 1}(x)$ coincides with the smallest integer $d_1$
for $L(\Bd, \ul\Bh)$  appearing in the expansion of $x$
in terms of $\SX_{\ul\Bh}$, where $\Bd = (d_1, \dots, d_6)$.  
The case $j = \ul 2$ is described similarly.
Let $\wt\SL_j(\infty)$ be the $\ZZ[q]$-submodule of $\ul\BU_q^-$ spanned by
$q^{-\ve_j(b_1)+1}b_1$, where $b_1 \in \bB_{\ul \Bh}$ (resp. $b_1 \in \bB_{\ul \Bh'}$)
if $j = \ul 1$ (resp. $j = \ul 2$). 
\par
We show the former statement.  The latter is proved similarly.
Since $\ul b = b(\ul\Bc',\ul\Bh')$ by assumption in 3.10, one can write as
\begin{equation*}
  L(\ul\Bc, \ul\Bh) = L(\ul\Bc', \ul\Bh') + \sum_{\Bd}a_{\Bd}L(\Bd, \ul\Bh'),
         \quad (a_{\Bd} \in q\BZ[q]).
\end{equation*}
We apply $f_{\ul 2}$ on both sides. Since
$f_{\ul 2}L(\ul\Bc', \ul\Bh') = [a'+1]L(\ul\Bc'^+, \ul\Bh')$, and  a similar formula
holds for $L(\Bd, \ul\Bh')$, we have
\begin{equation*}
\tag{3.11.2}  
f_{\ul 2}L(\ul\Bc,\ul\Bh) \equiv q^{-a'}L(\ul\Bc'^+, \ul\Bh') \mod q\wt\SL_{\ul 2}(\infty). 
\end{equation*}
Since $\ul b^{\circ} = \d b(\ul\Bc^{\bullet +}, \ul\Bh')$, we have
\begin{equation*}
  q^{-a'}L(\ul\Bc^{\bullet +}, \ul\Bh') \equiv \d q^{-a'}L(\ul\Bc^{\bullet + \bullet}, \ul\Bh)
              \mod q^{-a'+1}\SL_{\ZZ}(\infty).
\end{equation*}
Combining this with (3.11.2), we have
\begin{equation*}
  f_{\ul 2}L(\ul\Bc, \ul\Bh) \equiv \d q^{-a'}L(\ul\Bc^{\bullet + \bullet}, \ul\Bh)
           \mod q\wt\SL_{\ul 2}(\infty) + q^{-a'+1}\SL_{\BZ}(\infty).
\end{equation*}
Since $\ve_{\ul 2}(\ul b^{\circ}) = a' +1$, the coefficient of $\ul b^{\circ}$ in
$q\wt\SL_{\ul 2}(\infty) + q^{-a'+1}\SL_{\BZ}(\infty)$ is contained in $q^{-a'+1}\ZZ[q]$.
Hence (3.11.1) holds. 
\end{proof}  

The following lemma is immediate from Lemma 3.7.

\begin{lem}  
Let $\ul b^{\circ} = b(\ul\Bc^{\bullet + \bullet}, \ul\Bh)$ be as in 3.10.
Then the PBW basis $L(\ul\Bc^{\bullet +\bullet}, \ul\Bh)$ corresponding to $\ul b^{\circ}$
is given as follows.

\begin{enumerate}
\item \ $f_{\ul 1}^{(a-1)}f_{\ul 1\ul 2}^{(b+1)}f_{\ul 1\ul 2\ul 2}^{(c)}f_{\ul 2}^{(d)}$
          \quad if $a > cd(a+c)/\a, \ a > c$,
\item \ $f_{\ul 1}^{(a)}f_{\ul 1\ul 2}^{(b-1)}f_{\ul 1\ul 2\ul 2}^{(c+1)}f_{\ul 2}^{(d)}$
          \quad if $a \le cd(a+c)/\a, \ ab(a+c) > d, \ a \le c$,           
\item \ $f_{\ul 1}^{(a)}f_{\ul 1\ul 2}^{(b)}f_{\ul 1\ul 2\ul 2}^{(c)}f_{\ul 2}^{(d+1)}$
          \qquad if $a \le cd(a+c)/\a, \ ab(a+c)\le d$.          
\end{enumerate}  
\end{lem}  

\begin{prop}  
Let $\ul f_2L(\ul\Bc, \ul\Bh) = \sum_{b_1 \in \ul\Bh}\xi_{b_1}b_1$ as in Lemma 3.11.
Then we have
\begin{equation*}
\tag{3.13.1}  
\xi_{\ul b^{\circ}} \in q^{-a'}(1 + q\ZZ[q]).
\end{equation*}  
In particular, $\d = 1$, and we have
$b(\ul\Bc^{\bullet + \bullet}, \ul\Bh) = b(\ul\Bc^{\bullet +}, \ul\Bh')$.
\end{prop}
\begin{proof}
Set $Z = f_{\ul 2}L(\ul\Bc, \ul\Bh)$, and let $y = L(\ul\Bc^{\bullet + \bullet}, \ul\Bh)$
be the PBW basis corresponding to $\ul b^{\circ} = b(\ul\Bc^{\bullet + \bullet}, \ul\Bh)$.
$y$ is given by (i) $\sim$ (iii) in Lemma 3.12.
Note that those three PBW bases are exactly the same as the PBW  bases appearing
in the expansion of $Z$ in Proposition 3.9 (i).
In each case, we show that the coefficient of the lowest term of $\xi_{\ul b^{\circ}}$
is equal to 1. Then by  Lemma 3.11, (3.13.1) holds, and by comparing it with (3.11.1),
the proposition will follow.   
\par
For each coefficient $A, B, C \in \BA$ in Proposition 3.9 (i), we denote by
$\eta_A, \eta_B, \eta_C$ the degree of the lowest term. We have
\begin{equation*}
\eta_A = -b, \qquad \eta_B = 2a - b - 2c, \qquad \eta_C = 2a-2c-d.
\end{equation*}

\par\medskip
{\bf Case (i). } \
$y$ is the first term of $Z$, and it is the smallest element in the total order in $\SX_{\ul\Bh}$. 
Hence by Proposition 3.9 (i), $\xi_{\ul b^{\circ}} = [b+1]$. In particular, the coefficient
of the lowest term of $\xi_{\ul b^{\circ}}$ is equal to 1. 
\par\medskip
{\bf Case (ii).} \
$y$ is the second term of $Z$, and in this case, there is a  possibility that
$\ul b^{\circ}$ appears in the expansion by $\bB_{\ul\Bh}$, of the first term
(no contribution from the third term).
If we can show that $\eta_B < \eta_A +1$, the coefficient of the lowest term of
$\xi_{\ul b^{\circ}}$ coincides with that of $B$, hence is equal to 1.
Thus it is enough to show that
\begin{equation*}
\tag{3.13.2}  
2a - b - 2c < -b + 1.
\end{equation*}  
By the condition in Lemma 3.12 (ii), we have $a \le c$, hence (3.13.2) holds.
\par\medskip
{\bf Case (iii). } \
$y$ is the third term of $Z$, and in this case, we have to consider the
contribution from  the first and the second terms. 
As in the case (ii), if we can show that $\eta_C < \eta_A +1, \eta_C < \eta_B +1$,
then the coefficient of $\xi_{\ul b^{\circ}}$ is equal to 1.
Hence it is enough to show that 
\begin{equation*}
2a - 2c -d < -b + 1, \qquad 2a - 2c -d < 2a-b-2c +1.
\end{equation*}  
This is equivalent to the condition that
\begin{equation*}
\tag{3.13.3}
  2(a-c) - (b-d) \le 0, \qquad b-d \le 0. 
\end{equation*}  
By Lemma 3.12 (iii), $(a,b,c,d)$ satisfies the relation
\begin{equation*}
\tag{3.13.4}  
a\a \le cd(a+c), \qquad ab \le d(a+c),
\end{equation*}
where $\a = ab + ad + cd$ (in the semifield notation).
The latter condition implies that $ab \le ad + cd = \min\{ad, cd\}$, hence
we have $b \le d$ and $ab \le cd$. In particular, the second relation of (3.13.3) holds.
If $a \le c$, the first relation of (3.13.3) holds, thus we may assume that $c \le a$. 
Then $\a = ab + (a + c)d = ab + cd = ab$. By the first condition in (3.13.4),
we have $a\a = a^2b \le c^2b$. This implies that $2a + b \le 2c + d$ in $\ZZ$, and
the first relation of (3.13.3) holds. 
Hence (3.13.3) holds. 
\end{proof}  

\para{3.14.}
We now compute $f_{\ul 1}L(\ul\Bc, \ul\Bh')$, and prove a similar property
as in Proposition 3.13 for it. In this case, the computation is more complicated.
Let $\vf$ be the piecewise
linear bijection given in (3.3.2).  Then $\vf\iv : (a,b,c,d) \mapsto (a',b',c',d')$
can be computed by reversing the operation in (3.3.1) as follows;
\begin{equation*}
\tag{3.14.1}  
a' = bcd/\a, \quad b' = \a^2/\ve, \quad c' = \ve/\a, \quad d' = ab^2c/\ve, 
\end{equation*}  
where
\begin{align*}
\a &= a(b+d) + cd = ab + ad + cd, \\
\ve &= ab(b+d) + d\a = ab^2 + ad^2 + cd^2 + abd.
\end{align*}

\par
Let $\ul\Bc = (a, b, c, d)$ and $\ul\Bc' = \vf\iv(\ul\Bc) = (a',b,'c',d')$.
$\vf(a'+1, b', c',d') = \ul\Bc^{\bullet +}$ can be computed in a similar way as Lemma 3.7.
By making use of it, we can show the following lemma, which is an analogue of Lemma 3.12.

\begin{lem}  
Let $\ul b^{\circ} = b(\ul\Bc^{\bullet + \bullet}, \ul\Bh')$. 
Then the PBW basis $L(\ul\Bc^{\bullet + \bullet}, \ul\Bh')$ corresponding to
$\ul b^{\circ}$ is given as follows.
\begin{enumerate}
\item \ $f_{\ul 2}^{(a-2)}f_{\ul 1\ul 2\ul 2}^{(b+1)}f_{\ul 1\ul 2}^{(c)}f_{\ul 1}^{(d)}$
\par\medskip\noindent  
if \  
$  \begin{cases}
    \a/(b +d) > bcd\a/(b+d)\ve, \\
    a > cd/(b+d), \\
    b \le d/1,
  \end{cases}
  \text{ or } \ 
    \begin{cases}
      \a/(b+d) \le bcd\a/(b +d)\ve,  \\
      ab(b +d)/\a > d,  \\
      b \le d/1,
    \end{cases}  
$
\par\medskip    
\item  \
  $f_{\ul 2}^{(a -1)}f_{\ul 1\ul 2\ul 2}^{(b)}f_{\ul 1\ul 2}^{(c+1)}f_{\ul 1}^{(d)}
  \ \text{ if } \ 
     \begin{cases}
       \a/(b +d)\le bcd\a/(b +d)\ve, \\
       ab(b +d)/\a > d, \\
       cd/(b +d) \le a \le cd1/(b+d),
     \end{cases}
  $
\par\medskip
\item \
$f_{\ul 2}^{(a)}f_{\ul 1\ul 2\ul 2}^{(b -1)}f_{\ul 1\ul 2}^{(c +2)}f_{\ul 1}^{(d)}
  \  \text{ if } \
       \begin{cases}
         \a/(b+d)\le bcd\a/(b+d)\ve, \\
         ab(b+d)/\a > d,  \\
         a < cd/(b+d), \\
         b > d,
       \end{cases}
$
\par\medskip
\item \
$ f_{\ul 2}^{(a)}f_{\ul 1\ul 2\ul 2}^{(b)}f_{\ul 1\ul 2}^{(c)}f_{\ul 1}^{(d+1)}
   \quad  \text { if } \  
        \begin{cases}
          \a/(b+d) \le bcd\a/(b+d)\ve, \\
          ab(b+d)/\a \le d, \\
          a \le cd/(b+d), \\
          b \le d,
        \end{cases}
        $
\end{enumerate}
where $d/1, d1$ are written by the semifield notation,
which correspond to $d -1, d+1$ in $\ZZ$. 
\end{lem}

The following result is an analogue of Proposition 3.13, but which is weaker than it.

\begin{prop}  
Let $f_{\ul 1}L(\ul\Bc, \ul\Bh') = \sum_{b_1 \in \bB_{\ul\Bh'}}\xi_{b_1}b_1$
be as in Lemma 3.11.  Let $\ul\Bc = (a,b,c,d)$.
Assume the condition $(*)$ ``$a = 0$ if $a = c$ and $b = d$''. 
Then we have
\begin{equation*}
\tag{3.16.1}  
\xi_{\ul b^{\circ}} \in q^{-a'}(1 + q\ZZ[q]).
\end{equation*}
In particular, $\d = 1$, and we have
$b(\ul c^{\bullet + \bullet}, \ul\Bh') = b(\ul \Bc^{\bullet +},\ul\Bh)$.   
\end{prop}
\begin{proof}
Set $Z = f_{\ul 1}L(\ul \Bc, \ul\Bh')$, and let $y = L(\ul\Bc^{\bullet + \bullet}, \ul\Bh')$
be the PBW basis corresponding to $\ul b^{\circ} = b(\ul\Bc^{\bullet + \bullet}, \ul\Bh')$.
$y$ is given as in (i) $\sim$ (iv) in Lemma 3.15.  Again, those four PBW bases are exactly
the same as the PBW bases appearing in the expansion of $Z$ in Proposition 3.9 (ii).
We prove the proposition in a similar way as Proposition 3.13.
\par
For each coefficient $A, \dots, D \in \BA$ in Proposition 3.9 (ii),  we denote by
$\eta_A, \dots, \eta_D$ the degree of the lowest term.  We have
\begin{align*}
\eta_A &= -2b, \\
\eta_B &= a - 2b -c -1, \\
\eta_C &= 2a-2b-2c,  \\
\eta_D &= 2a -2c -2d.  
\end{align*}

We show in each case, the coefficient of the lowest term of $\xi_{\ul b^{\circ}}$
is equal to 1.
\par\medskip
{\bf Case (i).}
$y$ is the first term of $Z$, and it is the smallest element in the total order
in $\SX_{\ul\Bh'}$.  Hence (3.16.1) holds.

\par\medskip
{\bf Case (ii).}
$y$ is the second term of $Z$.  We show that $\eta_B < \eta_A + 1$, namely,
\begin {equation*}
\tag{3.16.2}  
a - 2b - c - 1 < -2b + 1 \Longleftrightarrow a -c < 2.
\end{equation*}       
\par
By Lemma 3.15 (ii), $(a,b,c,d)$ satisfies the condition
\begin{equation*}
\begin{cases}
  \a/(b+d) \le bcd\a/(b+d)\ve, \\
  ab(b+d)/\a > d, \\
  cd/(b+d) \le a \le cd1/(b+d)
\end{cases}  \Longleftrightarrow
\begin{cases}
  \ve \le bcd, \\
  ab(b+d) > d\a, \\
  cd \le a(b+d) \le cd1,
\end{cases}  
\end{equation*}
where $\a = a(b+d) + cd$, $\ve = ab(b+d) + d\a$.
First assume that $d \le b$.  Then $b + d = d$, and the third condition implies
that $cd \le ad \le cd1$, namely, $c \le a \le c+1$ in $\ZZ$.  Hence
$a- c \le 1 < 2$.
Next assume that $b < d$.  Then $b + d = b$, and $\a = ab + cd$.
We have $\ve = ab^2 + abd + cd^2 = ab^2 + cd^2$ since
$ab^2 < abd$. By the first condition, we have
$ab^2 + cd^2 \le bcd < cd^2$. This implies that $ab^2 < cd^2$.
On the other hand, the second condition implies that $ab^2 > d\a= d(ab + cd) > ab^2 + cd^2$,
hence $ab^2 > cd^2$. This is absurd, and the case $b < d$ does not occur.
Hence (3.16.2) holds, and (3.16.1) follows. 

\par\medskip
{\bf Case (iii).}  \
$y$ is the third term of $Z$.    
We show that $\eta_C < \eta_A +1, \eta_C < \eta_B +1$, namely,    
\begin{equation*}
\tag{3.16.3}  
2a - 2b - 2c < -2b + 1, \qquad 2a - 2b - 2c < a - 2b - c.
\end{equation*}  
This is equivalent to $2a - 2c < 1$ and $a < c$, hence enough to
see that $a < c$. 
By Lemma 3.15 (iii), $(a,b,c,d)$ satisfies the condition 
\begin{equation*}
\begin{cases}
\a/(b+d)\le bcd\a/(b+d)\ve, \\
ab(b+d)/\a > d, \\
a \le cd/(b+d)1, \\ 
b > d.
\end{cases}
\end{equation*}
Since $b > d$, we have $b + d = d$.  Then the third condition 
$a(b+d)1\le cd$ implies that $a + 1 \le c$ in $\ZZ$.   
Hence $a < c$, and (3.16.3) holds.  We obtain (3.16.1). 

\par\medskip
{\bf Case (iv). } \
Let $y$ be the fourth term of $Z$. We show that
$\eta_D < \eta_A +1, \eta_D < \eta_B +1, \eta_D < \eta_C +1$, namely,
\begin{align*}
2a - 2c - 2d &< -2b +1, \\  
2a - 2c - 2d &< a -2b -c, \\
2a - 2c - 2d &< 2a - 2b - 2c +1, \\
\end{align*}  
This is equivalent to the relations
\begin{align*}
2(a-c) + 2(b-d) &< 1, \\
\tag{3.16.4}
(a-c) + 2(b-d) &< 0, \\  
2(b-d) &< 1.   
\end{align*}
By Lemma 3.15 (iv), $(a,b,c,d)$ satisfies the condition
\begin{equation*}
\begin{cases}
\a/(b+d) \le bcd\a/(b+d)\ve, \\
ab(b+d)/\a \le d, \\  
a \le cd/(b+d), \\
b \le d,
\end{cases}
\Longleftrightarrow
\begin{cases}
\ve \le bcd, \\  
ab(b+d) \le d\a, \\
a(b+d) \le cd, \\
b \le d,
\end{cases}  
\Longleftrightarrow
\begin{cases}
\ve \le bcd, \\  
ab^2 \le  d\a, \\
ab \le cd, \\
b \le d.
\end{cases}  
\end{equation*}  
By the fourth condition, we have $b \le d$, hence $2(b-d) < 1$,
so the third relation in (3.16.4) holds.  
By the third condition, we have $ab \le cd$, hence $a+b \le c+ d$ in $\ZZ$,
and so $2(a-c) + 2(b-d) \le 0$.  Hence the first relation in (3.16.4) holds. 
By the second condition, we have $ab^2 \le d\a$. 
Since $\a = ab + ad + cd = \min\{ a+b, a+d, c+d \}$ in $\ZZ$, this condition
implies that
\begin{equation*}
a + 2b \le d + \min\{ a+b, a+d, c+d\} \le c + 2d.
\end{equation*}
Hence $(a-c) + 2(b-d) \le 0$.  If $(a-c) + 2(b-d) \ne 0$, then
the second relation holds.  So we assume that $(a-c) + 2(b-d) = 0$.
We have $a - c \le d-b$, and so $(a-c) = 2(d-b) \le d-b$. Since $b \le d$
by the third relation in (3.16.4), this implies that
$d = b$, hence $a = c$.  
\par
Now by the assumption (*), we have $a = 0$. 
Then by Proposition 3.9 (ii), we have
\begin{equation*}
  Z = f_{\ul 1}L(\ul\Bc, \ul\Bh') =
        Cf_{\ul 1\ul 2\ul 2}^{(b-1)}{f'}_{\ul 1\ul 2}^{(c+2)}f_1^{(d)}
        + Df_{\ul 1\ul2 \ul 2}^{(b)}{f'}_{\ul 1\ul 2}^{(c)}f_{\ul 1}^{(d+1)},
\end{equation*}
and $y$ is the second term of $Z$. 
Thus it is enough to check the relation $\eta_D < \eta_C +1$,
namely, the third relation in (3.16.4).  
But we have already verified this in the previous discussion. 
Hence (3.16.4) holds, and (3.16.1) follows. 
\par\smallskip
The proposition follows from the above discussion. 
Note that the assumption (*) is used only in the case (iv). 
\end{proof}

We are now ready to prove (2.12.1), namely,
\begin{thm}  
Assume that $\ul\BU_q^-$ is of type $B_2$.  Then we have $\bB_{\ul\Bh} = \bB_{\ul\Bh'}$.   
\end{thm}  
\begin{proof}
For $\g = \sum_{j \in J}n_j\a_j \in \ul Q$, set $|\g| = \sum_{j \in J}|n_j|$.   
We prove, for $\g \in - \ul Q_+$, 
\begin{equation*}
\tag{3.17.1}  
  (\bB_{\ul\Bh})_{\g} = (\bB_{\ul \Bh'})_{\g}
\end{equation*}  
by induction on $|\g|$. Thus it is enough to show that
$b(\ul\Bc, \ul\Bh) = b(\ul\Bc^{\bullet}, \ul\Bh')$ for any $b \in (\bB_{\ul\Bh})_{\g}$.
This is clear if $\g = 0$. We assume that $\g \ne 0$, and that  
(3.17.1) holds for $\g'$ such that $|\g'| < |\g|$.
\par
We follow the discussion in 2.11. We know that $\wt\CB = \bB_{\ul\Bh} \sqcup -\bB_{\ul\Bh}$
is the canonical signed basis of $\ul\BU_q^-$. For $j \in J$,
let $F_j$ be the Kashiwara operator on $\ul\BU_q^-$ defined in (2.11.3). 
Take $b \in (\bB_{\ul\Bh})_{\g}$.   Assume that $\ve_{\ul 2}(b) > 0$.
Then by (2.11.2), there exists $b'\in \bB_{\ul\Bh}$ such that $F_{\ul 2}(b') = b$.
Write $b' = b(\ul\Bc, \ul\Bh)$.  Since $b' \in (\bB_{\ul\Bh})_{\g'}$ with
$|\g'| < |\g|$, we have $b(\ul\Bc, \ul\Bh) = b(\ul\Bc^{\bullet}, \ul\Bh')$
by induction hypothesis. Then $F_{\ul 2}(b') = b(\ul\Bc^{\bullet +}, \ul\Bh')$
by (2.11.4).  
By Proposition 3.13, we have
$b(\ul\Bc^{\bullet +}, \ul\Bh') = b(\ul\Bc^{\bullet +\bullet}, \ul\Bh)$.
Hence $b = b(\ul\Bc^{\bullet +}, \ul\Bh') = b(\ul\Bc^{\bullet +\bullet},\ul\Bh)$.  
\par
Next assume that $\ve_{\ul 2}(b) = 0$. Since $b \ne 1$, by (2.11.1) there exists
$b'$ such that $F_{\ul 1}(b') = b$. Write $b' = b(\ul\Bc, \ul\Bh')$ with
$\ul\Bc = (a,b,c,d)$.
By a similar discussion as above, we have $b = b(\ul\Bc^{\bullet +}, \ul\Bh)$. 
If $a \ne c$ or $b \ne d$, then 
$b = b(\ul\Bc^{\bullet + }, \ul\Bh) = b(\ul\Bc^{\bullet + \bullet}, \ul\Bh')$ by
Proposition 3.16. 
Hence we may assume that $a = c$ and $b = d$, namely $\ul\Bc = (a,b,a,b)$.
In this case, it is easy to see that
$\ul\Bc^{\bullet} = (b,a,b,a)$ and $\ul\Bc^{\bullet + \bullet} = (a,b,a,b+1)$.  
Since $\ve_{\ul 2}(b) = 0$ and $b = \pm b(\ul\Bc^{\bullet + \bullet}, \ul\Bh')$,
we have $a = 0$.
Again by applying Proposition 3.16 and by (2.11.4), we have 
$b = F_{\ul 1}(b') = b(\ul\Bc^{\bullet +}, \ul\Bh) = b(\ul\Bc^{\bullet + \bullet}, \ul\Bh')$. 
The theorem is proved.
\end{proof}  

\par

\par\vspace{1.5cm}
\noindent
T. Shoji \\
School of Mathematica Sciences, Tongji University \\ 
1239 Siping Road, Shanghai 200092, P.R. China  \\
E-mail: \verb|shoji@tongji.edu.cn|

\par\vspace{0.5cm}
\noindent
Z. Zhou \\
School of Digital Science, Shanghai Lida Univesity \\ 
1788 Cheting Road, Shanghai 201608, P.R. China  \\
E-mail: \verb|forza2p2h0u@163.com|

\end{document}